\documentclass[reqno,12pt]{article}
\usepackage{latexsym}
\usepackage{amsmath}
\usepackage{amssymb}
\usepackage{bbm}
\usepackage{amsthm}
\usepackage{color}

\usepackage{graphicx}
\usepackage{epsfig}

\definecolor{Red}{rgb}{1,0,0}

 0 \DeclareMathAlphabet{\mathbfit}{OT1}{cmr}{bx}{it}

\newcommand{\bea}{\begin{eqnarray*}}
\newcommand{\eea}{\end{eqnarray*}}
\newcommand{\be}{\begin{eqnarray}}
\newcommand{\ee}{\end{eqnarray}}
\newcommand{\beq}{\begin{equation}}
\newcommand{\eeq}{\end{equation}}

\newtheorem{thm}{Theorem}%[section]

\newtheorem{rmk}[thm]{Remark}
\newtheorem{lem}[thm]{Lemma}
\newtheorem{fact}{Fact}
\newtheorem{prop}[thm]{Proposition}
\newtheorem{defi}[thm]{Definition}

\begin{document}

\title{The speed of biased random walk among random conductances} 

\author{
{\small Noam Berger, Nina Gantert, Jan Nagel}\\
}

\maketitle

\begin{abstract}We consider biased random walk among iid, uniformly elliptic conductances on $\mathbb{Z}^d$, and investigate the monotonicity of the velocity as a function of the bias. It is not hard to see that if the bias is large enough, the velocity is increasing as a function of the bias.
Our main result is that if the disorder is small, i.e. all the conductances are close enough to each other, the velocity is always strictly increasing as a function of the bias, see Theorem \ref{thm:monotone}.
A crucial ingredient of the proof is a formula for the derivative of the velocity, which can be written as a covariance, see 
Theorem \ref{thm:derivative}: it follows along the lines of the proof of the Einstein relation in \cite{einstein}.
On the other hand, we give a counterexample showing that for iid, uniformly elliptic conductances, the velocity is not always increasing as a function of the bias.
More precisely, if $d=2$ and
if the conductances take the values $1$ (with probability $p$) and $\kappa$ (with probability $1-p$) and $p$ is close enough to $1$ and $\kappa$ small enough, the velocity is {\sl not} increasing as a function of the bias, see Theorem \ref{thm:nonmonotone}.
\end{abstract}

{\bf Keywords :}
{Random walk in random environment, random conductances, effective velocity

\smallskip

{\bf MSC 2010:} {
60K37; %processes in random environments
60J10; %Markov chains (discrete-time Markov processes on discrete state spaces)
60K40 %Other physical applications of random processes
%60G50 %Sums of independent random variables; random walks
}

\section{Introduction}

As a model for transport in an inhomogeneous medium, one may consider a biased random walk on a supercritical percolation cluster.
The model goes back, to our best knowledge, to Mustansir Barma and Deepak Dhar, see \cite{barmadhar} and \cite{dhar}.
They conjecured the following picture for the velocity (in the direction of the bias) as a function of the bias.
The velocity is increasing for small values of the bias, then it is decreasing to $0$ and remains $0$ for large values of the bias, see 
Figure \ref{pic:speed2} below.
Here, the zero velocity regime is due to ``traps'' in the environment which slow down the random walk.
It was proved by \cite{sznitman2003anisotropic} and by \cite{BGP2003speed} that the velocity is indeed zero if the bias is large enough, while it is strictly positive for small values of the bias. Later, Alexander Fribergh and Alan Hammond were able to show that there is a sharp transition, i.e. there is a critical value of the bias such that the velocity is zero if the bias is larger, and strictly positive if the bias is smaller than the critical value, see \cite {FriHam}.

The velocity of biased random walk among iid, uniformly elliptic conductances is always strictly positive, this was proved by Lian Shen in \cite{shen2002}.
 A criterion for ballisticity in the elliptic, but not uniformly elliptic case can be found in \cite{friberghRCM}.
It is interesting to ask about monotonicity in the uniformly elliptic case. In the following, $v_1(\lambda)$ denotes the component of the velocity in the direction of the bias, precise definitions are below.
In the homogeneous medium (i.e. if the conductances are constant), the velocity can be computed and the picture is as in Figure \ref{pic:speed1}.
\begin{figure}[h] 
\centering
   \includegraphics[width=5cm]{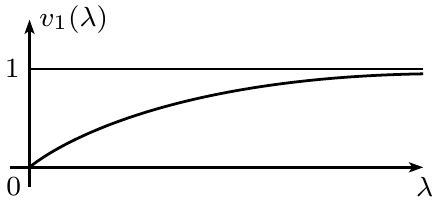}  
   \caption{Speed of biased simple random walk}
   \label{pic:speed1} 
 \end{figure}
For the biased random walk on a (supercritical) percolation cluster, the conjectured picture is as in Figure \ref{pic:speed2}.
\begin{figure}[h] 
\centering
   \includegraphics[width=5cm]{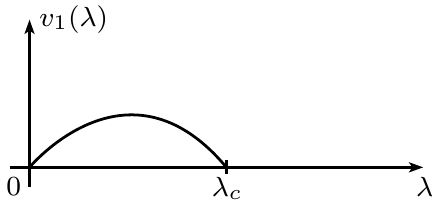}
   \caption{Conjectured speed of biased random walk on percolation clusters}
   \label{pic:speed2}
 \end{figure}
Now, in our case of iid, uniformly elliptic conductances, the picture should be ``in between'' the other two cases.
If the conductances are close enough to each other, we show that the speed is increasing, hence the picture is as in Figure \ref{pic:speed1}.
Under the assumptions of Theorem \ref{thm:nonmonotone}, we show that the speed is not increasing and Figure \ref{pic:speed3} is the simplest picture which agrees with our results.
\begin{figure}[h] 
\centering
   \includegraphics[width=5cm]{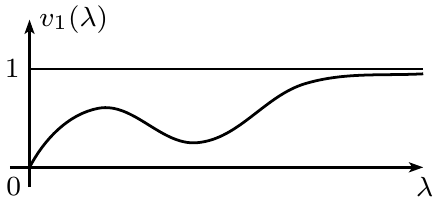}
   \caption{Conjectured speed of biased random walk under the assumptions of Theorem \ref{thm:nonmonotone}}
   \label{pic:speed3}
 \end{figure}
However, we only prove parts of this picture: we know that for $\lambda \to \infty$, the velocity is increasing and goes to $1$, see Fact \ref{fct:mon} below, and we show that the velocity is not increasing for all values of the bias, see Theorem \ref{thm:nonmonotone}.

Finally, let us mention some results for biased random walks on supercritical Galton-Watson trees with a bias pointing away from the root.
This model can be seen as a ``toy model'' for the percolation case, when the lattice is replaced by a tree.
For biased random walks on (supercritical) Galton-Watson trees with leaves, the velocity shows the same regimes as for
biased random walks on percolation clusters: it is zero if the bias is larger than a critical value, while it is strictly positive if the bias is less (or equal) than the critical value.
This transition  was proved by \cite{LPP} and the critical value has an explicit description, see \cite{LPP}. In particular, if the tree has leaves, the velocity can not be an increasing function of the bias.
For biased random walks on supercritical Galton-Watson trees without leaves the velocity is conjectured to be increasing, but despite recent progress, see \cite{BFS2014}, \cite{Aidekon2014}, this conjecture is still open.

Let us now give more precise statements and a description of our results.
For two neighboring vertices $x$ and $y$ in $\mathbb{Z}^d$ with $d\geq 2$, assign to the edge between $x$ and $y$ a nonnegative \emph{conductance} $\omega(x,y)$. The random walk among the conductances $\omega$ starting at $x_0$ and with bias $\lambda \geq 0$ (in direction $e_1 = (1,0,0, \ldots , 0)$)  is then the Markov chain $(X_n)_{n\geq 0}$ with law $P^{x_0}_{\omega,\lambda}$, defined by the transition probabilities 
\begin{align*}
P^{x_0}_{\omega,\lambda}\big( X_{n+1} = y | X_n=x \big) = \frac{\omega(x,y)e^{ \lambda(y-x)\cdot e_1}}{\sum_{z\sim x} \omega(x,z) e^{\lambda(z-x)\cdot e_1}}
\end{align*}
for $x\sim y$. (Here we write $x\sim y$ if $x,y$ are neigboring vertices, and we write $w\cdot z$ for the scalar product of two vectors 
$w, z \in \mathbb{R}^d$). The corresponding expectation is written as $E^{x_0}_{\omega,\lambda}$. The Markov chain $(X_n)_{n\geq 0}$ is reversible with respect to the measure
\begin{align*}
\pi(x) = \sum_{z\sim x} \omega(x,z) e^{\lambda(x+z)\cdot e_1} .
\end{align*}
When the collection of conductances $\omega$ is random with law $P$, we call $(X_n)_{n\geq 0}$ \emph{random walk among random conductances} and $P^{x_0}_{\omega,\lambda}$ the \emph{quenched law}. 
 $\mathbb{P}^{x_0}_\lambda = \int P^{x_0}_{\omega,\lambda}(\cdot ) P(d\omega)$ is called the \emph{annealed law} and we write $\mathbb{E}^{x_0}_\lambda $ for the corresponding expectation. If $x_0=0$ we omit the superscripts. 
In this paper we study properties of the limiting velocity
\begin{equation}\label{speed-def}
v(\lambda)=\lim_{n\to \infty} \frac{X_n}{n}.
\end{equation} 
Frequently, we focus on the speed in direction $e_1$ and set $v_1(\lambda) = v(\lambda) \cdot e_1$. In particular, we are interested in the monotonicty of $v_1$ as a function of the bias $\lambda$. Although increasing $\lambda$ increases the local drift to the right at every point, it is not clear at all that this results in a higher effective velocity. As mentioned above, this conclusion is known to be false for a biased random walk on a percolation cluster, which corresponds to conductances $\omega(x,y)\in \{0,1\}$. As shown by \cite{FriHam}, the speed is positive for $\lambda$ smaller than some critical value $\lambda_c>0$, but increasing the bias further will give zero speed. If we assume the conductances to be uniformly elliptic, that is, there exists a $\delta\in(0,1)$ such that 
\begin{equation}\label{UE}
1-\delta \leq \omega(x,y) \leq 1+\delta ,
\end{equation}
then \cite{shen2002} showed that the limit in \eqref{speed-def} exists $\mathbb{P}_\lambda$ almost surely, does not depend on $\omega$, and there is no zero speed regime: $v_1(\lambda)>0$ for all $\lambda >0$. From now on, we assume 

\medskip

\textbf{Assumption (A)}
The conductances are iid and uniformly elliptic, i.e. they satisfy \eqref{UE}.

\medskip

Note that \eqref{UE} is equivalent to the usual uniform ellipticity saying that the conductances are bounded above and bounded away from $0$: we may multiply all the conductances by a constant factor, resulting in the same transition probabilities.

\medskip

\begin{fact}\label{fct:inf}
%\textbf{Fact 1.}
$\lim_{\lambda \to \infty} v_1(\lambda)=1$.
\end{fact}
%\medskip

\begin{fact}\label{fct:mon}
%\textbf{Fact 2.}
There exists a $\lambda_c=\lambda_c(\delta)$ such that $v_1$ is strictly increasing on $[\lambda_c,\infty)$. 
\end{fact}
\medskip

Fact \ref{fct:inf} follows from a coupling with a random walk in a homogeneous environment, as 
\begin{align} \label{lowerbounde1}
P_{\omega,\lambda}\big( X_{n+1} = x+e_1 | X_n=x \big) \geq \frac{e^\lambda}{(2d-1)\frac{1+\delta}{1-\delta}+e^\lambda} ,
\end{align}
%\red{ich bekomme
%$$
%P_{\omega,\lambda}\big( X_{n+1} = x+e_1 | X_n=x \big) \geq \frac{e^\lambda}{e^\lambda + 
%(2d-2)\frac{1+\delta}{1-\delta}+e^{-\lambda}\frac{1+\delta}{1-\delta}}
%$$
%}
which goes to 1 as $\lambda \to \infty$. Fact \ref{fct:mon} was proven by \cite{BFS2014} for the biased random walk on a Galton-Watson tree without leaves (where an upper bound for $\lambda_c$ can be explicitly computed), the same arguments yield the analogous result for the conductance model, when the conductances are bounded away from 0 and $\infty$. A sketch of the proof will be given in Section 2. We remark that $\lambda_c(\delta)$ may be chosen decreasing in $\delta$.

Our first main result shows that in the low disorder regime, when $\delta$ is close to 0, $v_1$ is increasing on $[0,\infty)$. That is, in the low disorder regime, Fact \ref{fct:mon} holds with $\lambda_c=0$.

\begin{thm}\label{thm:monotone}
Assume (A). There exists a $\delta_0\in (0,1)$, such that if $1-\delta_0\leq \omega(x,y) \leq 1+\delta_0$ whenever $x\sim y$, then $v_1$ is strictly  increasing.
\end{thm}

On the other hand, outside the low disorder regime, there is in general no monotonicity, in particular, uniform ellipticity of the conductances does not imply monotonicity of the speed.

\begin{thm} \label{thm:nonmonotone}
Assume (A) and $d = 2$.
Define the environment law by 
\begin{align*}
P(\omega(0,e)=1)=p=1-P(\omega(0,e)=\kappa)
\end{align*}
for $p\in(0,1)$ and $\kappa>0$. Then, for $p$ close enough to 1 and $\kappa$ close enough to 0, there exist $\lambda_1<\lambda_2$ such that
\begin{align*}
v_1(\lambda_1)>v_1(\lambda_2) .
\end{align*}
\end{thm}

To prove Theorem \ref{thm:monotone}, we show that the derivative of the speed is strictly positive, where the derivative can be expressed as the covariance of two processes. For this, we define
\begin{align}
M_n & = X_n - \sum_{k=0}^{n-1} E_{\omega,\lambda}^{X_k}[X_1-X_0] ,\label{processM} \\
N_n & = X_n - nv(\lambda) \label{processN}.
\end{align} 
We show in Proposition \ref{prop:clt} below that under $\mathbb{P}_\lambda$, the $2d$-dimensional process $\frac{1}{\sqrt{n}} (M_n,N_n)$ converges in distribution to a Gaussian limit $(M,N)$.

\begin{thm}\label{thm:derivative}
Assume (A). For any $\lambda>0$, $v$ is differentiable at $\lambda$ with
\begin{align*}
v'(\lambda) = \operatorname{Cov}_{\lambda}(M,N) e_1 .
\end{align*}
\end{thm}

\begin{rmk}
The statement in Theorem \ref{thm:derivative} is true for $\lambda =0$ as well - this is the Einstein relation proved in \cite{einstein}. In particular, $\lambda \to v_1(\lambda)$ is a continuous function. The continuity of $v_1$ may seem obvious, but to our best knowledge, it has not been proved for a biased random walk on a percolation cluster, and not even for biased random walk on Galton-Watson trees.
\end{rmk}

\section{A general coupling}
\label{sec:coupling}

After a suitable enlargement of our probability space, let $U_0,U_1,\dots$ be a sequence of independent random variables with a uniform distribution on $[0,1]$, independent of $\omega$. Let us denote the joint law of the $U_k$ and $\omega$ by $\mathbb{P}$, with expectation $\mathbb{E}$. We will construct a coupling of quenched laws for different environments and different values of the bias, letting $U_k$ determine the movement at time $k$. Given an environment $\omega$ and $\lambda\geq 0$, define 
\begin{align*}
p_{\omega,\lambda} (x,e) = P_{\omega,\lambda}(X_1=x+e| X_0=x) 
\end{align*}
and, with $e_{k}=-e_{k-d}$ for $d+1 \leq k \leq 2d$,
let $q_{\omega,\lambda}(x,0)=0$ and for $1\leq k\leq 2d$,
\begin{align*}
q_{\omega,\lambda} (x,k) = \sum_{j=1}^{k} p_{\omega,\lambda} (x,e_j) .
\end{align*}
Now, given two environments $\omega_1$ and $\omega_2$ and biases $\lambda_1$ and $\lambda_2$ we can define processes $X_n^{(\omega_1,\lambda_1)}$ and $X_n^{(\omega_2,\lambda_2)}$ by setting 
\begin{align*}
X_{n+1}^{(\omega_i,\lambda_i)}-X_n^{(\omega_i,\lambda_i)} = e_k 
\quad \text{ iff } \quad 
q_{\omega_i,\lambda_i} (X_n,k-1) < U_n \leq q_{\omega_i,\lambda_i} (X_n,k) 
\end{align*}
for $i=1,2$. Then the marginal of $(X_n^{(\omega_i,\lambda_i)})_n$ is the original quenched law $P_{\omega_i,\lambda_i}$. In the one-dimensional case this coupling also shows the monotonicity of the speed for any ellipticity constant, since then $\lambda_1\leq \lambda_2$ implies $X_n^{(\omega,\lambda_1)}\leq X_n^{(\omega,\lambda_2)} $.  To give a short justification of Fact \ref{fct:mon}, we additionally introduce for $\lambda_s>0$ the one-dimensional process 
\begin{align*}
Y_n = \sum_{k=0}^{n-1} \left( 2\cdot \mathbbm{1}\{U_k \leq \tfrac{e^{\lambda_s}}{e^{\lambda_s}+(2d-1)\beta} \} -1\right),
\end{align*}
where $\beta = \frac{1+\delta}{1-\delta}$. 
Assume $\lambda_s > \log \beta + \log (2d-1)$, then $Y_n$ is a simple random walk with drift to the right.
From the lower bound \eqref{lowerbounde1}, we see that if $Y_n$ moves to the right and $\lambda_s<\lambda_i$, then $X_n^{(\omega_i,\lambda_i)}$ moves to the right. This allows us to consider so-called super-regeneration times $\tau_k, k\geq 1$ (introduced by \cite{BFGH}) where $\tau_1$ is the infimum over all times $n\geq 1$ with
\begin{align*}
\max_{k<n-1} Y_k < Y_{n-1}<Y_n< \min_{k>n}Y_k ,
\end{align*} 
and inductively $\tau_{n+1} = \tau_1 \circ \theta_{\tau_n}+ \tau_n$ (here $\theta_k$ denotes the time shift, i.e. 
$\theta_k(Y_n)_{n\geq 0} = (Y_{k+n})_{n\geq 0}$). Since the increments of $Y_n$ are a lower bound for the increments of $X_n^{(\omega_i,\lambda_i)}$ in direction $e_1$, $\tau_1$ is a regeneration time for the process $X_n = X_n^{(\omega_i,\lambda_i)}$, provided $\lambda_i >\lambda_s$. 
More precisely,
\begin{align*}
\max_{k<\tau_n-1} X_k\cdot e_1 < X_{\tau_n-1}\cdot e_1 < X_{\tau_n}\cdot e_1 < \min_{k>\tau_n}X_k\cdot e_1 ,
\end{align*} 
Unlike in \cite{BFS2014}, we require an additional step to the right in order to decouple the environment seen by the random walker. By classical arguments, the sequence $(X_{\tau_k}^{(\omega,\lambda)} - X_{\tau_{k-1}}^{(\omega,\lambda)},\tau_k-\tau_{k-1})_{k\geq 2}$ is an iid sequence under $\mathbb{P}$, and the marginal is equal to the distribution of $(X_{\tau_1},\tau_1)$, conditioned on the event $R=\{Y_n>0 \text{ for all } n\geq 1\}$. 
%\red{Stimmt das? M\"usste es nicht $R=\{Y_1= 1, Y_n>1  \text{ for all } n\geq 2\}$ sein?}
Moreover, 
\begin{align*}
v(\lambda) = \frac{\mathbb{E}[X_{\tau_1}^{(\omega,\lambda)}| R]}{\mathbb{E}[\tau_1|R]}
\end{align*}
for any $\lambda>\lambda_s$. Fact \ref{fct:mon} follows then if we can show for $\lambda_s$ large enough and $\lambda > \lambda_s$,
\begin{align} \label{superreg1}
\mathbb{E}\left[ (X_{\tau_1}^{(\omega,\lambda+\varepsilon)}-X_{\tau_1}^{(\omega,\lambda)})\cdot e_1 | R\right] >0 ,
\end{align}
for any $\varepsilon>0$. Following the arguments of \cite{BFS2014}, this is implied by the following observations:
\begin{itemize}
\item When $Y_n$ moves to the right, both $X_n^{(\omega,\lambda)}$ and $X_n^{(\omega,\lambda+\varepsilon)}$ move to the right.
\item When $Y_n$ moves to the left for the first time, then
% the increments of $X_n^{(\omega,\lambda)}$ and $X_n^{(\omega,\lambda+\varepsilon)}$ differ for the first time, then
\begin{align*} 
(X_n^{(\omega,\lambda+\varepsilon)}-X_n^{(\omega,\lambda)})\cdot e_1\geq 0
\end{align*} 
and, given that $Y_n$ moves to the left for the first time at time $n$, with probability larger than some $p_0>0$, 
\begin{align*} 
(X_n^{(\omega,\lambda+\varepsilon)}-X_n^{(\omega,\lambda)})\cdot e_1> 0 .
\end{align*} 
\item When until time $\tau_1$ the process $Y_n$ took $k$ steps to the left, the increments of $X_n^{(\omega,\lambda)}$ and $X_n^{(\omega,\lambda+\varepsilon)}$ could differ at most $k$ times. 
\item When until time $n$ the increments of $X_n^{(\omega,\lambda)}$ and $X_n^{(\omega,\lambda+\varepsilon)}$ were different exactly $k$ times, then
\begin{align*} 
(X_{\tau_1}^{(\omega,\lambda+\varepsilon)}-X_{\tau_1}^{(\omega,\lambda)})\cdot e_1> -2(k-1)
\end{align*}
\item Let $D_k$ be the event that until time $\tau_1$, $Y_n$ did $k$ steps to the left and for some $n\leq \tau_1$, $X_n^{(\omega,\lambda+\varepsilon)}-X_n^{(\omega,\lambda)}\neq 0$. Then 
\begin{align}\label{BFSproof2}
\mathbb{E}\left[ (X_{\tau_1}^{(\omega,\lambda+\varepsilon)}-X_{\tau_1}^{(\omega,\lambda)})\cdot e_1 | R\right]
\geq p_0 \mathbb{P}(D_1|R) - \sum_{k\geq 2} 2(k-1) \mathbb{P}(D_k|R) .
\end{align}
For $\lambda_s$ large enough, the right hand side of \eqref{BFSproof2} is positive, which follows analogously to the proof in \cite{BFS2014} of positivity of display (4.1) therein.  
\end{itemize}

%To prove Theorem \ref{monoton}, we need to couple the random walk in environment $\omega$ to a random walk in a homogeneous environment $\bar{\omega}$, where $\bar{\omega}(x,y)=1$ for all nearest neighbors $x,y$. In this case, we may explicitly compute 

\section{Differentiating the speed} \label{sec:diff}

Theorem \ref{thm:derivative} is a consequence of the two following results.
For simplicity, we will omit integer parts.
\medskip

\begin{thm}\label{thm:diff1} Let $\lambda_0>0$, $\alpha > 1$ and $t_\lambda=\alpha\cdot (\lambda-\lambda_0)^{-2}$, then
\begin{align*}
\lim_{\lambda \to \lambda_0} \frac{\tfrac{1}{t_\lambda}\mathbb{E}_\lambda[X_{t_\lambda}] - v(\lambda_0)}{\lambda-\lambda_0} = \operatorname{Cov}_{\lambda_0}(M,N)\cdot e_1 .
\end{align*}
\end{thm}
\medskip

\begin{thm} \label{thm:diff2} Let $t_\lambda$ be as in Theorem \ref{thm:diff1}. There exists a $C>0$, such that for any $\alpha>1$, 
\begin{align*}
\limsup_{\lambda\to \lambda_0} \left| \frac{\tfrac{1}{t_\lambda}\mathbb{E}_\lambda[X_{t_\lambda}] - v(\lambda)}{\lambda-\lambda_0} \right| \leq \frac{C}{\sqrt{\alpha}} .
\end{align*}
\end{thm}

\medskip

\subsection{Regeneration times} \label{sec:regeneration}

The proof of Theorem \ref{thm:diff1} and Theorem \ref{thm:diff2} relies on a regeration structure for the process $(X_n)_n$, which decomposes the trajectory into 1-dependent increments with good moment bounds. For $h\in \mathbb{R}$, we let 
\begin{align*}
\mathcal{H}_h = \{x \in \mathbb{Z}^d |\, x \cdot e_1 = \lfloor h \rfloor \}
\end{align*}
denote the hyperplane with first coordinate $\lfloor h \rfloor$ and 
\begin{align*}
T_h = \inf \{n\geq 0|\, X_n \in \mathcal{H}_h \}
\end{align*}
be the first hitting time of $\mathcal{H}_h$.
The regeneration times $\tau_k$, $k\geq 1$ are then hitting times $T_{mL/\lambda}$, after which the random walk never visits $\mathcal{H}_{(m-1)L/\lambda}$ again and the displacement $X_{T_{mL/\lambda}}-X_{T_{(m-1)L/\lambda}}$ can be decoupled from the environment in $\{x \in \mathbb{Z}^d |\, x \cdot e_1 \leq \lfloor h\rfloor \}$. The detailed construction of the sequence $(\tau_k)_k$ can be found in \cite{einstein}, for the sake of brevity we only summarize here the consequences in the following lemma. We remark that the moment bounds are stated in \cite{einstein} only for $\lambda\in (0,\lambda_u)$ for some small $\lambda_u>0$, but the proof works actually for any bounded, positive $\lambda$. 
\begin{rmk}
Note that the $(\tau_k)_k$ are not the same as the super-regeneration times in Section \ref{sec:coupling} (which were also denoted by $(\tau_k)_k$)  but in order to be consistent with \cite{BFGH} and \cite{einstein}, we keep this notation. 
\end{rmk}

\begin{lem}\label{lem:regen}
Under $\mathbb{P}_\lambda$, the sequence
\begin{align*}
\left( (X_{k+1}-X_k)_{\tau_n\leq k < \tau_{n+1}}, \tau_{n+1}-\tau_n \right)_{n\geq 1}
\end{align*}
is a stationary 1-dependent sequence. Moreover, for any $\lambda_1>0$ there are constants $c,C >0$, such that for all $\lambda \in (0,\lambda_1]$ we have
\begin{align}\label{tauexpbounds}
\mathbb{E}_\lambda [\exp(c\lambda^2\tau_1)] \leq C, \quad  \quad \mathbb{E}_\lambda [\exp(c\lambda^2(\tau_2-\tau_1))] \leq C
\end{align}
and
\begin{align*}
\mathbb{E}_\lambda [\exp(c\lambda||X_{\tau_1}||)] \leq C, \quad  \quad \mathbb{E}_\lambda [\exp(c\lambda||X_{\tau_2}-X_{\tau_1}||)] \leq C .
\end{align*}

\end{lem}

\medskip

We also have a lower bound for the inter-regeneration time (see (21) in \cite{einstein}), where for any $\lambda_1>0$ there is a constant $c>0$, such that 
\begin{align}\label{reglowerbound}
\mathbb{E}_\lambda [\lambda^2(\tau_2-\tau_1)] \geq c  
\end{align}
for all $\lambda \in (0,\lambda_1]$.
If \eqref{UE} is satisfied with $\delta \leq \frac{1}{2}$, $c$ and $C$ in Lemma \ref{lem:regen} and in \eqref{reglowerbound} can be chosen only depending on the dimension.
Using the exponential moment estimates on the regeneration times, it follows that in order to study the convergence in distribution of $\frac{1}{\sqrt{n}} (M_n,N_n)$, it suffices to consider 
\begin{align*}
\frac{1}{\sqrt{\tau_n}}\big( M_{\tau_n},N_{\tau_n}\big). 
\end{align*}
To this subsequence, we may apply the functional central limit theorem for sums of 1-dependent random variables, see \cite{Bi56} to obtain the following result.

\medskip

\begin{prop}\label{prop:clt}
For any $\lambda>0$, the process $\big( \frac{1}{\sqrt{n}} (M_{\lfloor tn\rfloor},N_{\lfloor tn\rfloor}); 0\leq t\leq 1\big)$ converges in distribution under $\mathbb{P}_\lambda$ to a $2d$-dimensional Brownian motion $(\widehat{M}_t, \widehat{N}_t )$.
We write $M$ for $\widehat{M}_1$ and $N$ for $\widehat{N}_1$.
\end{prop}

\medskip

\begin{lem} \label{lem:mombound}
For any $p\in \mathbb{N}$ and $\lambda_1>0$ there exists a $C_p>0$ depending only on $p$, $\lambda_1$, the dimension $d$, and the ellipticity constant $\delta$, such that for any $0<\lambda<\lambda_1$,
\begin{align*}
\mathbb{E}_{\lambda} \left[ \max_{0\leq k\leq n/\lambda^2} \left|\left| \lambda X_k \right|\right|^p\right] \leq C_pn^p.
\end{align*}
%\begin{align*}
%\limsup_{n\to \infty} \mathbb{E}_{\lambda} \left[ \left|\left| \frac{1}{\sqrt{n}} (X_n-nv(\lambda)) \right|%\right|^4\right] \leq C
%\end{align*}
\end{lem}

\begin{proof}
The lemma follows from the proof of Lemma 8 in \cite{einstein}, 
noting that the constant $C_p$ there can be chosen
depending only on $p$, an upper bound for $\lambda$, the dimension $d$, and the ellipticity constant $\delta$.
\end{proof}

\subsection{Proof of Theorem \ref{thm:diff1}}

The arguments in this section are inspired by \cite{lebrost1994} where a weak form of the Einstein relation was proved for a large class of models.
Let us abbreviate $\bar{\lambda}=\lambda-\lambda_0$ and begin by writing, with $t = t_\lambda = 
\alpha/ \bar{\lambda}^2$,
\begin{equation}\label{massaendern}
\frac{\tfrac{1}{t}\mathbb{E}_\lambda[X_{t}] - v(\lambda_0)}{\lambda-\lambda_0}
= \mathbb{E}_\lambda \left[ \tfrac{\bar{\lambda}}{\alpha} \left( X_{t} - t\cdot v(\lambda_0) \right)\right]
= \mathbb{E}_{\lambda_0} \left[ \tfrac{\bar{\lambda}}{\alpha} \left( X_{t} - t\cdot v(\lambda_0) \right)\frac{dP_{\omega,\lambda}}{dP_{\omega,\lambda_0}}(X_s; 0\leq s\leq t)\right]
\end{equation}
as an expectation with respect to the reference measure $\mathbb{P}_{\lambda_0}$. For a nearest-neighbor path $(x_1,\dots ,x_m)$, we have
\begin{align*}
\frac{dP_{\omega,\lambda}}{dP_{\omega,\lambda_0}}(x_1,\dots ,x_m)
= \prod_{k=1}^m \frac{p_{\omega,\lambda}(x_{k-1},x_{k}-x_{k-1})}{p_{\omega,\lambda_0}(x_{k-1},x_{k}-x_{k-1})}
= \prod_{k=1}^m e^{\bar{\lambda} (x_k-x_{k-1})\cdot e_1} \frac{\sum_{|e|=1}e^{\lambda_0 e\cdot e_1} \omega(x_{k-1},x_{k-1}+e)}
{\sum_{|e|=1}e^{\lambda e\cdot e_1} \omega(x_{k-1},x_{k-1}+e)}.
\end{align*}
Now write in the denominator $e^{\lambda e\cdot e_1}=e^{\bar{\lambda} e\cdot e_1}e^{\lambda_0 e\cdot e_1}$ and expand the first exponential $e^{z}=1+z+z^2/2+r_1(z)$ with $|r_1(z)|\leq |z|^3$ for $|z|\leq 1$ to get
\begin{align*}
& \frac{dP_{\omega,\lambda}}{dP_{\omega,\lambda_0}}(x_1,\dots ,x_m) \\
& = \exp \left\{ \bar{\lambda} x_m\cdot e_1 - \sum_{k=1}^m \log \left( 1+ \bar{\lambda}d_{\omega,\lambda_0}(x_{k-1}) + \frac{1}{2} \bar{\lambda}^2 d_{\omega,\lambda_0}^{(2)}(x_{k-1}) + r_1(\bar{\lambda})\right) \right\} ,
\end{align*}
where we wrote 
\begin{align*}
d_{\omega,\lambda_0} (x) = \frac{\sum_{|e|=1} \omega(x,x+e)e^{\lambda_0 e\cdot e_1} e\cdot e_1}{\sum_{|e|=1} \omega(x,x+e)e^{\lambda_0 e \cdot e_1}} = E_{\omega,\lambda_0}^{x}[(X_1-X_0)\cdot e_1]
\end{align*}
for the local drift in direction $e_1$ and 
\begin{align*}
d_{\omega,\lambda_0}^{(2)} (x) = \frac{\sum_{|e|=1} \omega(x,x+e)e^{\lambda_0e\cdot e_1}(e \cdot e_1)^2}{\sum_{|e|=1} \omega(x,x+e)e^{\lambda_0 e\cdot e_1}} = E_{\omega,\lambda_0}^{x}[((X_1-X_0)\cdot e_1)^2]
\end{align*}
for the expected squared displacement. Expanding the logarithm as $\log(1+z) = z-z^2/2+r_2(z)$ with $|r_2(z)|\leq |z|^3$ for $|z|\leq 1/2$, we obtain 
\begin{align*}
\exp\left\{ \bar{\lambda} x_m\cdot e_1-\sum_{k=1}^m \left( \bar{\lambda} d_{\omega,\lambda_0}(x_{k-1})+\frac{\bar{\lambda}^2}{2}( d_{\omega,\lambda_0}^{(2)} (x_{k-1})- d_{\omega,\lambda_0} (x_{k-1})^2)+h(\bar{\lambda}) \right) \right\}
\end{align*}
where the function $h$ satisfies $|h(z)| \leq c|z|^3$ if $|z|\leq 1/2$. If we set now $m= t= \alpha/\bar{\lambda}^2$, this yields 
\begin{align} \label{density}
& \quad G_{\omega,\lambda_0}(\bar{\lambda},t): = \frac{dP_{\omega,\lambda}}{dP_{\omega,\lambda_0}}(X_k; 0\leq k\leq t)\\ \notag
& = \exp\left\{ \bar{\lambda} \left( X_{\alpha/\bar{\lambda}^2}\cdot e_1 - \sum_{k=1}^{\alpha/\bar{\lambda}^2} d_{\omega,\lambda_0} (X_{k-1}) \right) - \frac{\bar{\lambda}^2}{2} \sum_{k=1}^{\alpha/\bar{\lambda}^2} \left(d_{\omega,\lambda_0}^{(2)} (X_{k-1}) - d_{\omega,\lambda_0} (X_{k-1})^2\right) + o(1) \right\}  .
\end{align}
By Proposition \ref{prop:clt}, $\bar{\lambda} \big( X_{\alpha/\bar{\lambda}^2}\cdot e_1 - \sum_{k=1}^{\alpha/\bar{\lambda}^2}d_{\omega,\lambda_0} (X_{k-1}) \big)$ converges in distribution to $\widehat{M}_\alpha\cdot e_1$. To infer the convergence of the complete expression for the density and to obtain convergence of the expectations in \eqref{massaendern},  we next show $L^p$-boundedness of the density.

Recall $G_{\omega,\lambda_0}(\bar{\lambda},t)$ in \eqref{density}, and let $p\geq 1$. Then
\begin{align*}
p \log G_{\omega,\lambda_0}(\bar{\lambda},t) 
& = p \bar{\lambda} X_t\cdot e_1 - p \sum_{k=1}^t \log \left( \frac{\sum_{|e|=1} \omega(X_{k-1},X_{k-1}+e)e^{\lambda e\cdot e_1} }{\sum_{|e|=1} \omega(X_{k-1},X_{k-1}+e)e^{\lambda_0 e \cdot e_1}}\right) \\
& = p\bar{\lambda} X_t \cdot e_1 - \sum_{k=1}^t \log \left( \frac{\sum_{|e|=1} \omega(X_{k-1},X_{k-1}+e)e^{(\lambda_0+p\bar{\lambda}) e\cdot e_1} }{\sum_{|e|=1} \omega(X_{k-1},X_{k-1}+e)e^{\lambda_0 e \cdot e_1}}\right) + R_{\omega,\lambda_0}(\bar{\lambda},t) \\
& = \log G_{\omega,\lambda_0}(p\bar{\lambda},t) + R_{\omega,\lambda_0}(\bar{\lambda},t),
\end{align*}
with a remainder term 
\begin{align*}
R_{\omega,\lambda_0}(\bar{\lambda},t) = \sum_{k=1}^t \log \left( \frac{\sum_{|e|=1} \omega(X_{k-1},X_{k-1}+e)e^{(\lambda_0+p\bar{\lambda}) e\cdot e_1} }{\sum_{|e|=1} \omega(X_{k-1},X_{k-1}+e)e^{\lambda_0 e \cdot e_1}} \right) - p \log \left( \frac{\sum_{|e|=1} \omega(X_{k-1},X_{k-1}+e)e^{\lambda e\cdot e_1} }{\sum_{|e|=1} \omega(X_{k-1},X_{k-1}+e)e^{\lambda_0 e \cdot e_1}} \right) .
\end{align*}
After expanding the exponential and then the logarithm as for \eqref{density}, we get
\begin{align*}
R_{\omega,\lambda_0}(\bar{\lambda},t)& = \sum_{k=1}^t \left( p\bar{\lambda}d_{\omega,\lambda_0} (X_{k-1}) +\frac{1}{2} p^2 \bar{\lambda}^2 (d_{\omega,\lambda_0}^{(2)} (X_{k-1})-d_{\omega,\lambda_0} (X_{k-1})^2) +o(p^2\bar{\lambda}^2)  \right) \\
& \qquad \quad -p \left( \bar{\lambda}d_{\omega,\lambda_0} (X_{k-1}) +\frac{1}{2}  \bar{\lambda}^2 (d_{\omega,\lambda_0}^{(2)} (X_{k-1})-d_{\omega,\lambda_0} (X_{k-1})^2) +o(\bar{\lambda}^2)  \right) \\
& \leq (p^2-p) \alpha + o(1) \leq p^2 \alpha +1
\end{align*} 
for $|\bar{\lambda}|$ smaller than some $\eta>0$. For such a choice of $\lambda$, 
\begin{equation}\label{mompbound}
\mathbb{E}_{\lambda_0} [G_{\omega,\lambda_0}(\bar{\lambda},t) ^p ] \leq \mathbb{E}_{\lambda_0} [G_{\omega,\lambda_0}(p\bar{\lambda},t)  ]e^{p^2\alpha +1} = e^{p^2\alpha +1} .
\end{equation}
Consequently, $(G_{\omega,\lambda_0}(\bar{\lambda},t))_{|\bar{\lambda}|\leq \eta}$ is uniformly bounded in $L^p(\mathbb{P}_{\lambda_0})$. Since this implies convergence of expectations, we get for the density \eqref{density}
\begin{align*}
\frac{dP_{\omega,\lambda}}{dP_{\omega,\lambda_0}}(X_k; 0\leq k\leq t) \xrightarrow[\bar{\lambda} \rightarrow 0 ]{d} \exp \left\{ \widehat{M}_\alpha\cdot e_1 - \tfrac{1}{2} \mathbb{E}_{\lambda_0}[(\widehat{M}_\alpha\cdot e_1)^2] \right\}
\end{align*} 
under $\mathbb{P}_{\lambda_0}$. By Proposition \ref{prop:clt}, we have also the weak convergence of the product
\begin{align*}
\tfrac{\bar{\lambda}}{\alpha} \left( X_{t} - t\cdot v(\lambda_0) \right)\frac{dP_{\omega,\lambda}}{dP_{\omega,\lambda_0}}(X_s; 0\leq s\leq t) \xrightarrow[\bar{\lambda} \rightarrow 0 ]{d} \frac{1}{\alpha} \widehat{N}_\alpha  \exp \left\{ \widehat{M}_\alpha \cdot e_1 - \tfrac{1}{2} \mathbb{E}_{\lambda_0}[(\widehat{M}_\alpha\cdot e_1)^2] \right\} .
\end{align*}
Moreover, this product is by Lemma \ref{lem:mombound} and the calculations above bounded in $L^2(\mathbb{P}_{\lambda_0})$. In particular, it is uniformly integrable and so the expectations converge as well, 
\begin{align*}
\mathbb{E}_\lambda \left[ \tfrac{\bar{\lambda}}{\alpha} \left( X_{t} - t\cdot v(\lambda_0) \right)\right]
\xrightarrow[\bar{\lambda} \rightarrow 0 ]{ }
\frac{1}{\alpha} \mathbb{E}_{\lambda_0} \left[ \widehat{N}_\alpha  \exp \{ \widehat{M}_\alpha\cdot e_1 - \tfrac{1}{2} \mathbb{E}_{\lambda_0}[(\widehat{M}_\alpha\cdot e_1)^2] \} \right] .
\end{align*}
By Girsanov's theorem, the limit is equal to the covariance $Cov_{\lambda_0}(M,N) e_1$ (recalling $M= \widehat{M}_1, N= \widehat{N}_1$).  \qed

\subsection{Proof of Theorem \ref{thm:diff2}}

Define $\gamma_n=\mathbb{E}_\lambda[\tau_n]$ and for $t>0$ fixed, let $n\geq 0$ be such that $\gamma_n \leq t < \gamma_{n+1}$. Then 
\begin{align*}
\left|\left| \frac{1}{t} \mathbb{E}_\lambda[X_t] - \frac{1}{\gamma_n} \mathbb{E}_\lambda[X_{\gamma_n}]\right|\right|
& \leq \frac{1}{t} \left|\left| \mathbb{E}_\lambda[X_t] -  \mathbb{E}_\lambda[X_{\gamma_n}]\right|\right|
+ \mathbb{E}_\lambda[||X_{\gamma_n}||] \left| \frac{1}{t} - \frac{1}{\gamma_n} \right| \\
& \leq \frac{\gamma_{n+1}-\gamma_n}{\gamma_n} + \gamma_n \frac{t-\gamma_n}{t\gamma_n} \\
& \leq 2\frac{\gamma_{n+1}-\gamma_n}{\gamma_n} \leq \frac{c}{n},
\end{align*}
by the moment bounds of Lemma \ref{lem:mombound}. Next, we have %for $\rho \in (0,1)$ we bound
\begin{align*}
\frac{1}{\gamma_n}  \left|\left| \mathbb{E}_\lambda[X_{\gamma_n}] - \mathbb{E}_\lambda[X_{\tau_n}]\right|\right| \leq \frac{1}{\gamma_n} \mathbb{E}_\lambda[(\tau_n-\gamma_n)^2]^{1/2} \leq C \frac{\sqrt{n}}{\gamma_n} \leq \frac{C}{\sqrt{n}} .
\end{align*}
%\frac{1}{\alpha_n}  \left|\left| \mathbb{E}_\lambda[X_{\alpha_n}] - \mathbb{E}_\lambda[X_{\tau_n}]\right|%\right| 
%& \leq \frac{1}{\alpha_n} \cdot \rho \alpha_n + \frac{1}{\alpha_n} \mathbb{E}_\lambda \left[ ||%X_{\tau_n}-X_{\alpha_n} || \mathbbm{1}\{ |\tau_n-\alpha_n|> \rho \alpha_n\} \right]\\
%& \leq \rho + \frac{1}{\alpha_n} \left( \mathbb{E}_\lambda \left[ |{\tau_n}-{\alpha_n} |^2\right] 
%\mathbb{P}_\lambda (|\tau_n-\alpha_n|> \rho \alpha_n)\right)^{1/2} \\
%& \leq \rho  + \frac{1}{\rho \alpha_n^2}  \mathbb{E}_\lambda \left[ |{\tau_n}-{\alpha_n} |^2\right] \\
%& \leq 
By the law of large numbers and stationarity of the inter-regeneration times, the speed is given by  %changed from epoch. But epoch means absolute time, not time that passed between two epochs.
\begin{align*}
v(\lambda) = \frac{\mathbb{E}_{\lambda}[X_{\tau_2}-X_{\tau_1}]}{\mathbb{E}_{\lambda}[\tau_2-\tau_1]}
= \frac{\mathbb{E}_{\lambda}[X_{\tau_n}-X_{\tau_1}]}{\mathbb{E}_{\lambda}[\tau_n-\tau_1]} = \frac{\mathbb{E}_{\lambda}[X_{\tau_n}] -\mathbb{E}_{\lambda}[X_{\tau_1}]}{\gamma_n - \gamma_1}
\end{align*}
such that we have
\begin{align*}
\left|\left| \frac{\mathbb{E}_\lambda [X_{\tau_n}]}{\gamma_n} - v(\lambda) \right|\right|
& \leq ||\mathbb{E}_\lambda[X_{\tau_n}]||\left| \frac{1}{\gamma_n}-\frac{1}{\gamma_n-\gamma_1}\right| 
			+ \frac{1}{\gamma_n-\gamma_1} \left|\left|\mathbb{E}_{\lambda}[X_{\tau_1}] \right| \right| \\
& \leq \gamma_n \frac{\gamma_1}{\gamma_n(\gamma_n-\gamma_1)} + \frac{\gamma_1}{\gamma_n-\gamma_1}\\
& \leq \frac{C}{n} .
\end{align*}
Putting the above estimates together, we get
\begin{align}\label{regenbound}
\left|\left|\frac{1}{t} \mathbb{E}_\lambda[X_t] - v(\lambda) \right|\right| \leq \frac{C}{\sqrt{n}} .
\end{align}
Recall that we set $\bar{\lambda} = \lambda - \lambda_0$ and $t=\alpha/\bar{\lambda}^2$.
Hence $t < \gamma_{n+1} \leq cn$, implying
\begin{align*}
\frac{1}{\sqrt{n}} \leq \frac{c}{\sqrt{t}} = \frac{c \bar{\lambda}}{\sqrt{\alpha}}.
\end{align*}
This and the inequality \eqref{regenbound} implies the estimate of Theorem \ref{thm:diff2}. \qed

\section{Monotonicity}

\subsection{Proof of Theorem \ref{thm:monotone}}

By Fact \ref{fct:mon}, it suffices to show (strict) monotonicity of $v_1$ on $[0,\lambda_c]$. We do this by showing that the derivative on this compact interval is strictly positive. More precisely, we compare $v_1'$ with $\bar{v}'_1$, where 
\begin{align*}
\bar{v}_1(\lambda) =\frac{e^\lambda-e^{-\lambda}}{e^{\lambda}+e^{-\lambda}+2d-2} 
\end{align*}
is the speed of the random walk in a homogeneous environment $\bar\omega$, where all conductances equal 1. Since $\bar{v}'_1$ is greater than some positive $\varepsilon_0$ on $[0,\lambda_c]$, positivity of $v'_1$ follows from 
\begin{align}\label{comparev}
\sup_{\lambda \in [0,\lambda_c]} |v'_1(\lambda) - \bar{v}'_1(\lambda)| < \varepsilon_0
\end{align}
for $\delta$ close enough to 0. Let us assume already $\delta \leq \tfrac{1}{2}$. In Section 2 we constructed a coupling $(X_n^{(\omega,\lambda)},X_n^{(\bar\omega,\lambda)})_n$ between the random walk in an original environment $\omega$ and a random walk in the homogeneous environment $\bar\omega$. To keep the notation simpler, we denote $X_n^{(\omega,\lambda)}$ again by $X_n$ and $X_n^{(\bar\omega,\lambda)}$ by $\bar{X}_n$. Furthermore, define analogously to \eqref{processM} and \eqref{processN} the processes $\bar{M}_n$ and $\bar{N}_n$ in the homogeneous environment. (Of course,  $\bar{M}_n = \bar{N}_n$).
The coupling guarantees then
\begin{align} \label{couplingerror}
P(X_n-X_{n-1} \neq \bar X_n - \bar X_{n-1})\leq C\delta ,
\end{align}
so if $\delta$ is sufficiently small, the two processes will take the same steps most of the time. 
By Theorem \ref{thm:derivative} and the moment bounds in Lemma \ref{lem:mombound}, we have
\begin{align*}
 & \quad v'_1(\lambda)-\bar{v}'_1(\lambda) \\ & = 
\lim_{n\to \infty} \frac{1}{n} \big[ \operatorname{Cov}_\lambda ( M_n,N_n)_{1,1}  - \operatorname{Cov}_\lambda ( \bar M_n,\bar N_n)_{1,1} \big]\\
& =  \lim_{n\to \infty} \frac{1}{n} \big[ \operatorname{Cov}_\lambda ( M_n-\bar M_n,N_n-\bar N_n)_{1,1} + \operatorname{Cov}_\lambda ( M_n-\bar M_n,\bar N_n)_{1,1} +  \operatorname{Cov}_\lambda( \bar M_n,N_n-\bar N_n)_{1,1}\big] . 
\end{align*}
By the Cauchy-Schwarz inequality, \eqref{comparev} will follow from the following bounds:
\begin{align}
\label{limsupA} \limsup_{n\to \infty} \frac{1}{n} \operatorname{Var}_\lambda (\bar N_n)_{1,1} &\leq C \\
\label{limsupB} \limsup_{n\to \infty} \frac{1}{n} \operatorname{Var}_\lambda (\bar M_n)_{1,1} &\leq C \\
\label{limsupC} \limsup_{n\to \infty} \frac{1}{n} \operatorname{Var}_\lambda (M_n - \bar M_n)_{1,1} &\leq C\delta \\
\label{limsupD} \limsup_{n\to \infty} \frac{1}{n} \operatorname{Var}_\lambda (N_n - \bar N_n)_{1,1} &\leq g(\delta) 
\end{align}
with $g$ a function independent of $\lambda$ and $\lim_{\delta\to 0} g(\delta)=0$ (In fact, all these are actual limits). The first two bounds \eqref{limsupA} and \eqref{limsupB} follow since $\bar N_n = \bar M_n$ is a  process in the homogeneous environment with iid increments uniformly bounded in $\lambda$ and $\delta$.

For \eqref{limsupC}, observe that $ (M_n-\bar M_n)\cdot e_1 $ is again a martingale with 
\begin{align*}
& \quad \mathbb{E}_\lambda \left[ \big( (M_n-\bar M_n)\cdot e_1 - (M_{n-1}-\bar{M}_{n-1})\cdot e_1\big)^2\right] \\
& \leq 2 \mathbb{E}_\lambda \left[ \big( (X_n-X_{n-1})\cdot e_1-(\bar{X}_n-\bar{X}_{n-1})\cdot e_1\big)^2 \right]
 + 2 \mathbb{E}_\lambda \left[ \big( d_{\omega,\lambda}(X_{n-1})\cdot e_1- d_{\bar\omega,\lambda}(\bar X_{n-1})\cdot e_1\big)^2 \right]
\end{align*}
where $d_{\omega, \lambda}(x)= E^x_{\omega,\lambda}[X_1-X_0]$.
By \eqref{couplingerror}, the first term is of order $\delta$. 
We have
\begin{align} \label{couplingerror2}
||d_{\omega, \lambda}(x) - d_{\bar\omega, \lambda}(x)||\leq C\delta ,
\end{align}
so that the second term is of order at most $\delta$ as well. Consequently, 
\begin{align*}
\limsup_{n\to \infty} \frac{1}{n} \mathbb{E}_\lambda\left[ \big( (M_n-\bar M_n)\cdot e_1 \big)^2 \right] \leq C\delta . 
\end{align*}

It remains to show \eqref{limsupD}. We decompose
\begin{align*}
 (N_n -\bar N_n)\cdot e_1 = \big(X_n -\bar X_n - n(v(\lambda)-\bar v(\lambda))\big) \cdot e_1 = (M_n-\bar M_n)\cdot e_1 + Z_n ,
\end{align*}
where 
\begin{align*}
Z_n = \sum_{k=0}^{n-1} \big(d_{\omega, \lambda}(X_{k}) - d_{\bar \omega, \lambda}(X_{k})\big)\cdot e_1 - (v_1(\lambda)-\bar{v}_1(\lambda)) . 
\end{align*}
We already know that the difference of the martingales is nicely bounded and it therefore suffices to bound
\begin{align}\label{relevantvariance}
\lim_{n\to \infty} \frac{1}{n}\mathbb{E}_\lambda [(Z_n)^2] & = \lim_{n\to \infty} \frac{\mathbb{E}_\lambda [(Z_{\tau_n})^2]}{\mathbb{E}_\lambda[\tau_n]} \notag  \\
& = \frac{\mathbb{E}_\lambda [(Z_{\tau_2}-Z_{\tau_1})^2]+ 2 \mathbb{E}_\lambda [(Z_{\tau_3}-Z_{\tau_2})(Z_{\tau_2}-Z_{\tau_1})]}{\mathbb{E}_\lambda[\tau_2-\tau_1]} \notag \\
& \leq 3 \frac{\mathbb{E}_\lambda [(Z_{\tau_2}-Z_{\tau_1})^2]}{\mathbb{E}_\lambda[\tau_2-\tau_1]} ,
\end{align}
where we used the fact that $(Z_{\tau_n}-Z_{\tau_{n-1}},\tau_n-\tau_{n-1})$ is a stationary 1-dependent sequence. In fact, by Jensen's inequality it suffices to bound \eqref{relevantvariance} with $Z_n $ replaced by $\xi_n$, when 
\begin{align*}
\xi_n = \sum_{k=0}^{n-1} \big(d_{\omega, \lambda}(X_{k}) - d_{\bar\omega, \lambda}(X_{k})\big)\cdot e_1 . 
\end{align*}
 The uniform bound \eqref{couplingerror2} gives
\begin{align} \label{easybound}
\frac{\mathbb{E}_\lambda [(\xi_{\tau_2}-\xi_{\tau_1})^2]}{\mathbb{E}_\lambda[\tau_2-\tau_1]} \leq C\delta^2 \frac{\mathbb{E}_\lambda [(\tau_2-\tau_1)^2]}{\mathbb{E}_\lambda[\tau_2-\tau_1]} \leq C\frac{\delta^2}{\lambda^2},
\end{align}
where we used \eqref{tauexpbounds} and \eqref{reglowerbound} for the last inequality. Of course, this bound blows up near $\lambda=0$, but it yields for any $\lambda_0>0$
\begin{align*}
\lim_{\delta \to 0} \sup_{\lambda \in [\lambda_0,\lambda_c]} |v'_1(\lambda) - \bar{v}'_1(\lambda)| = 0 .
\end{align*}
Now suppose there are environment measures compatible with our a priori bound $\delta\leq \tfrac{1}{2}$ such that the speed is not monotone on $[0,\infty)$. If none of these measures satisfies the uniform ellipticity assumption with some smaller $\delta'>0$, we may just choose such an ellipticity constant to exclude these measures. If there exists a sequence $P^{(n)}$ of environment measures with ellipticity constants $\delta_n\to 0$ and such that the speed is not monotone, then we may find a sequence of $\lambda_n>0$ with $|v_1'(\lambda_n)-\bar v_1'(\lambda_n)|\geq \varepsilon_0$. By the bound \eqref{easybound}, we have necessarily $\lambda_n\to 0$. 
To complete the proof we show that such a sequence cannot exist. 
%as $\delta_n\to 0$ and $\lambda_n\to 0$, we have
%\begin{align}\label{boundnear0}
%\lambda^2(\xi_{\tau_2}-\xi_{\tau_1})^2\rightarrow  0 .
%\end{align} 
%In order to make more sense out of the limit $\delta_n\to 0$, let us introduce a coupling of environment measures for different ellipticity constants. For an environment $\omega\sim P$ define
%\begin{align*}
%\omega^\delta(x,y) = 2\delta (\omega(x,y)-1)+1.
%\end{align*} 
%Then, if $\omega$ has ellipticity constant $\tfrac{1}{2}$, $\omega^\delta$ has ellipticity constant $\delta $. In particular, $\omega^0=\bar\omega$. Let then $P_{\omega,\lambda}$ be the path measure of the random walks $(X_n^{(\delta)})_n$ in the environment $\omega^\delta $ with bias $\lambda$. 

\begin{lem}\label{lem:deltato0} 
For any sequence of environment measures $P^{(n)}$ with ellipticity constants $\delta_n\to 0$ and any sequence $\lambda_n$ with $\lambda_n\to 0$,
\begin{align*}
\lim_{n\to \infty} \mathbb{E}^{(n)}_{\lambda_n} [\lambda_n^2(\xi_{\tau_2}-\xi_{\tau_1})^2] = 0 .
\end{align*}
\end{lem}

%\textbf{Proof: }
\begin{proof}
To simplify notation, let us drop some of the indices $n$, in particular we write $\lambda$ for $\lambda_n$. We have for $i=1,2$
\begin{align*}
\mathbb{E}^{(n)}_\lambda [\lambda^2(\xi_{\tau_i})^2] & \leq \sum_{N=1}^\infty \mathbb{E}^{(n)}_\lambda\left[ \lambda^2 ( \xi_{N/\lambda^2}^*)^2 \mathbbm{1}_{\{ (N-1)/\lambda^2\leq \tau_i <N/\lambda^2 \} } \right] \\
& \leq \sum_{N=1}^\infty \mathbb{E}^{(n)}_\lambda\left[ \lambda^3 | \xi_{N/\lambda^2}^*|^3 \right]^{2/3} \mathbb{P}^{(n)}_\lambda (\tau_i \geq N/\lambda^2)^{1/3} ,
\end{align*}
with 
\begin{align*}
\xi_{N/\lambda^2}^* = \max_{0\leq k\leq N/\lambda^2} \xi_{k}  .
\end{align*}
By the moment bound for $\tau_i$, 
\begin{align*}
\mathbb{E}^{(n)}_\lambda [\lambda^2(\xi_{\tau_i})^2]  \leq C \sum_{N=1}^\infty \mathbb{E}^{(n)}_\lambda\left[ \lambda^3 | \xi_{N/\lambda^2}^*|^3 \right]^{2/3} e^{-cN } .
\end{align*}
Using the decomposition of $\xi_k=(M_k-\bar M_k)\cdot e_1 + (X_k-\bar X_k)\cdot e_1$ into a martingale term with bounded increments and the process $X_k$, Doob's inequality and the bound in Lemma \ref{lem:mombound} implies
\begin{align} \label{uniformmombound}
 \mathbb{E}^{(n)}_\lambda\left[ \lambda^4 |\xi_{N/\lambda^2}^*|^4 \right] \leq C N^4 , 
\end{align}
such that the assertion of the lemma will follow once we show that for every $N$,%from
\begin{align*}
\lim_{n \to \infty} \mathbb{E}^{(n)}_\lambda\left[ \lambda^3 |\xi_{N/\lambda^2}^*|^3 \right] = 0 .
\end{align*}  
We write the expectation with respect to the unbiased measure,
\begin{align*}
\mathbb{E}^{(n)}_\lambda\left[ \lambda^3 |\xi_{N/\lambda^2}^* |^3 \right] = \mathbb{E}^{(n)}_0\left[ \lambda^3 |\xi_{N/\lambda^2}^* |^3 G(\omega^{(n)},\lambda,N/\lambda^2)\right]
\end{align*}
with 
\begin{align*}
G(\omega,\lambda,m) = \frac{dP_{\omega,\lambda}}{dP_{\omega,0}}(X_k; 0\leq k\leq m),
\end{align*}
and $\omega^{(n)}$ distributed according to $P^{(n)}$.
We know that 
\begin{align*}
G(\omega^{(n)},\lambda,N/\lambda^2) = \exp \left( \lambda M_{N/\lambda^2}\cdot e_1 - \tfrac{1}{2} \mathbb{E}^{(n)}_0[(\lambda M_{N/\lambda^2}\cdot e_1)^2] + o(\lambda)\right)
\end{align*}
with an error term uniformly in $\delta$. Since $\delta$ and the distribution of $\omega^{(n)}$ is now varying with $\lambda$, $M_{N/\lambda^2}$ is now a triangular array of martingales. Thanks to the fact that all increments are uniformly (in $\delta$ and $\lambda$) bounded, the CLT for arrays of martingales yields
\begin{align*}
G(\omega^{(n)},\lambda,N/\lambda^2) \xrightarrow[n \to \infty]{d} e^{\widehat M_N\cdot e_1-\tfrac{1}{2}E[(\widehat M_N\cdot e_1)^2]}
\end{align*}
with $\widehat M_N$ a Gaussian random variable. Again, this convergence is complemented by a good moment bound, see \eqref{mompbound},
\begin{align*}
\mathbb{E}^{(n)}_0[G(\omega^{(n)},\lambda,N/\lambda^2)^p] \leq e^{p^2\tfrac{N}{2}+1} 
\end{align*}
for all $n$ and $p\geq 1$. Therefore, it suffices to show
\begin{align}\label{deltalambdalimit}
\lambda \xi_{N/\lambda^2}^* \xrightarrow[n \to \infty]{} 0
\end{align}
in probability. Until now we tacitly ignored that $d_{\omega,\lambda}(x)$ in the definition of $\xi_n=\xi_n(\lambda)$ depends on $\lambda$, but by the bound
\begin{align*}
||(d_{\omega,\lambda}(x)- d_{\bar\omega,\lambda}(x)) -(d_{\omega,0}(x)-d_{\bar\omega,0}(x))|| \leq C\delta \lambda
\end{align*}
we have
\begin{align*}
\lambda|\xi_{N/\lambda^2}^*(\lambda)-\xi_{N/\lambda^2}^*(0)| \leq CN\delta ,
\end{align*}
%so that from now on we can consider $\xi_{n}= \sum_{k=1}^{n} f_{k,\delta}$ with $f_{k,\delta}$ independent of $\lambda$. 
%In this framework, the claim \eqref{deltalambdalimit} is immediate from the following lemma. It is an adaptation of Lemma 2.4 in \cite{klo2012fluctuations}, which is based on Proposition 3.3 
%of \cite{kesten}, to our setting.

Therefore, it suffices to show that $\lambda\xi_{N/\lambda^2}^*(0)$ goes in probability to zero as $n$ goes to infinity. Recall that since for $\lambda=0$ the local drift in the environment $\bar\omega$ is zero, i.e.
$d_{\bar\omega,0}(x) = 0, \forall x$,  we get in fact
\[
\xi_n(0) = \sum_{k=0}^{n-1} d_{\omega,0}(X_{k})\cdot e_1 
\]
Lemma \ref{lem:kest} (with $L$ of that lemma set to be $N/\lambda^2$) below shows that
\[
\mathbb{E}_0^{(n)}\left[\left.\xi_{N/\lambda^2}^*(0)\right.^2
\right]
\leq CN\lambda^{-2}\delta,
%\mathbb{E} \left[ \sup_{0\leq n\leq N} \left(\sum_{k=0}^{n-1} d_\omega(X_{k})\right)^2 \right] \leq C N \delta
\]
so
\[
\mathbb{E}_0^{(n)}\left[\left(\lambda\xi_{N/\lambda^2}^*(0)\right)^2
\right]
\leq CN\delta
\]
which goes to zero as $n$ goes to infinity and then $\delta=\delta^{(n)}\to 0$.
\end{proof}

The next lemma is now all that is missing. The lemma is an adaptation of Lemma 2.4 in \cite{klo2012fluctuations}, which itself is based on the main idea of Proposition 3.3 of \cite{kesten}, to our setting.

\begin{lem}\label{lem:kest}
There exists a constant $C>0$ depending only on the dimension, such that for all $L\geq 1$ and $\delta\leq \tfrac{1}{2}$, we have, with $d_\omega(\cdot) = d_{\omega, 0}(\cdot)$,
\begin{align*}
\mathbb{E}_0 \left[ \sup_{0\leq n\leq L} \left|\left|\sum_{k=0}^{n-1} d_\omega(X_{k})\right|\right|^2 \right] \leq C L \delta .
\end{align*}  
\end{lem}

\textbf{Proof:} Recall that the environment measure $Q$ with
\begin{align*}
\frac{dQ}{dP}(\omega) = Z^{-1} \sum_{|e|=1} \omega(0,e)
\end{align*}
is stationary, reversible and ergodic for the process $(\widehat\omega_n)_n$ of the environment seen from the particle (see \cite{klo2012fluctuations} and \cite{einstein} for the definition of $(\widehat\omega_n)_n$ and some properties).  If $\delta \leq \tfrac{1}{2}$, the density satisfies $c\leq |\frac{dQ}{dP}(\omega)|\leq C$ with positive constants $c,C$ depending only on the dimension.
Therefore we may consider expectation with respect to $Q\times P_\omega$, which we denote by $\mathbb{E}_Q$. Under this measure, 
\begin{align*}
M_n = (X_n-\bar{X}_n) - (X_0-\bar{X}_0) - \sum_{k=0}^{n-1} d_{\omega}(X_k)
\end{align*}  
is a martingale with respect to the filtration $\mathcal{F}_n = \sigma(\{\widehat\omega_0,\ldots, \widehat\omega_n \})$. Since by time reversal, for any $n\ge 1$, the sequence 
\begin{align*}
\big( (X_1-X_0)-(\bar{X}_1-\bar{X}_0), \dots ,(X_n-X_{n-1})-(\bar{X}_n-\bar{X}_{n-1}), \widehat\omega_0,\ldots, \widehat\omega_n \big)
\end{align*}
has the same distribution as 
\begin{align*}
\big( (\bar{X}_n-\bar{X}_{n-1})- (X_n-X_{n-1}), \dots ,(\bar{X}_1-\bar{X}_{0})- (X_1-X_{0}), \widehat\omega_n,\ldots, \widehat\omega_0 \big)
\end{align*}
under  $Q_0\times P_\omega$, we have that 
\begin{align*}
M^-_n = (X_{L-n}-\bar{X}_{L-n}) - (X_L-\bar{X}_L) - \sum_{k=0}^{n-1} d_{\omega}(X_{L-k})
\end{align*}
is a martingale with respect to the filtration $\mathcal{F}_n^- = \sigma(\{\widehat\omega_L,\ldots, \widehat\omega_{L-n} \})$.
Noting that
\begin{align*}
M_L^- - M_{L-n}^- = (X_0-\bar{X}_0) - (X_{n}-\bar{X}_{n}) - \sum_{k=1}^n d_\omega(X_k) ,
\end{align*}
we get 
\begin{align*}
M_n+M_L^- - M_{L-n}^- = -2\sum_{k=0}^{n-1} d_\omega(X_k) + d_\omega(X_0)-d_\omega(X_n) .
\end{align*}
Therefore,
\begin{align*}
\mathbb{E}_Q \left[ \sup_{0\leq n\leq L} \left|\left|\sum_{k=0}^{n-1} d_\omega(X_{k})\right|\right|^2 \right]
\leq
\frac{1}{4} \mathbb{E}_Q \left[ \sup_{0\leq n\leq L} ||M_n+M_L^- - M_{L-n}^- - d_\omega(X_0)+d_\omega(X_n)||^2\right] .
\end{align*}
The lemma follows then from Doob's inequality, since $|d_\omega(x)|\leq C\delta $ and 
\begin{align*}
\mathbb{E}_Q \left[||M_L||^2\right] + \mathbb{E}_Q \left[||M_L^-||^2\right] \leq CL\delta .
\end{align*}

\qed

\subsection{Proof of Theorem \ref{thm:nonmonotone}}

The proof follows the arguments of \cite{BGP2003speed}, where the speed of biased random walk on a percolation cluster is studied. Note that the environment measure with 
\begin{align*}
P(\omega(0,e)=1)=p=1-P(\omega(0,e)=\kappa)
\end{align*}
generates a percolation graph consisting of the edges with conductance 1, connected by $\kappa$-edges. So if $p>\tfrac{1}{2}$ and $\kappa$ small enough, we would expect the random walk to behave like the random walk on the percolation cluster for most times, with short excursions along $\kappa$-edges. In analogy with the percolation case, we say in this section that an edge $\{x,y\}$ is open if $\omega(x,y)=1$ and (infinite) cluster will mean the (infinite) cluster connected by open edges.

We choose a bias $\lambda_1$, such that the random walk on the percolation cluster has a positive speed and show
\begin{align} \label{nonmonotonicitylowerbound}
v_1(\lambda_1)\geq c_0 , 
\end{align}
for a positive $c_0$ independent of $\kappa$. On the other hand, for a larger bias $\lambda_2$, chosen such that the random walk on the percolation cluster has zero speed, we show 
\begin{align} \label{nonmonotonicityupperbound}
v_1(\lambda_2)\leq c_0/2,
\end{align}
for $\kappa$ sufficiently small. The combination of these two bounds yield the statement of Theorem \ref{thm:nonmonotone}.

\subsubsection{A lower bound for $v_1(\lambda_1)$}

Denote the infinite cluster connected by open edges by $I$.

\begin{defi}\label{def:goodpoint}
A point $x\in \mathbb{Z}^2$ is good, if there exists an infinite path $x=x_0,x_1,x_2,\dots$ such that for all $k\geq 1$
\begin{itemize}
\item[(i)] $|(x_k-x_{k-1})\cdot e_2|=1$ and $(x_k-x_{k-1})\cdot e_1 =1$, 
\item[(ii)] the edges $\{x_{k-1},x_{k-1}+e_1\},\{x_{k-1}+e_1,x_k\}$ are open.
\end{itemize} 
\end{defi}

Let $J$ be the set of good vertices. We say a vertex $x$ is bad, if $x\in I$ and $x$ is not good. Connected components of $I\setminus J$ are called traps. For a vertex $x$, let $T(x)$ be the trap containing $x$ (being empty if $x$ is good). The length of the trap of $x$ is
\begin{equation*}
L(x)=\sup\{(y-z)\cdot e_1:\, y,z\in T(x) \}
\end{equation*}
and the width is
\begin{equation*}
W(x)=\sup\{(y-z)\cdot e_2:\, y,z\in T(x) \} .
\end{equation*}
If $T(x)$ is empty, then we take $L(x)=W(x)=0$. The following estimate is Lemma 1 in \cite{BGP2003speed}.

\begin{lem}\label{lem:trapsize}
For every $p\in (\tfrac{1}{2}, 1)$ there exists $\alpha=\alpha(p)$ such that
$P(L(0)\geq n)\leq \alpha^n$ and $P(W(0)\geq n)\leq \alpha^n$ for
every $n$. Further, $\lim_{p\to 1}\alpha(p) = 0$. 
\end{lem}

Let $\mathcal{H}(n)$ be the $\sigma$-algebra generated by the history of the random walk until time $n$, i.e., 
${\mathcal H}(n) = \sigma(\{X_0 = 0, X_1, X_2, \ldots ,X_n\})$.  
Let $P_{\omega,\lambda}^{{\mathcal H}(n)}$ be the conditional 
distribution of
$P_{\omega,\lambda}$ given ${\mathcal H}(n)$, and $\mathbb{P}_{\lambda}^{{\mathcal H}(n)}$ be 
the conditional 
distribution of
$\mathbb{P}_{\lambda}$ given ${\mathcal H}(n)$.
Define $\tau_n(h)=\min\{i>n :\, X_i\cdot e_1=h\}$. The following estimate is essential in the proof of the lower bound. 

\begin{lem}\label{lem:preconduc}
There exists $D'=D'(\lambda)$ such that for
every $\ell\geq 1$ and for every configuration $\omega$ such that $x$ is
a good point, 
\begin{equation*}
P_{\omega,\lambda}^{{\mathcal H}(n)}(\tau_n(x\cdot e_1-\ell)\leq\tau_n(x\cdot e_1+\ell/3)|X_n=x )<
D'e^{-\lambda \ell/3}.
\end{equation*}
\end{lem}

\textbf{Proof:} Consider the box $B = x + [-\ell , \ell/3]\times [ - e^{\lambda \ell} ,  e^{\lambda \ell}  ]$ with right face $B^+ = x+ \{\ell/3\}\times [ - e^{\lambda \ell} ,  e^{\lambda \ell}  ]$. From the general theory of electrical networks, see \cite{doylesnell} or \cite{LP:book}, we have the inequality
\begin{align*}
P_{\omega,\lambda}^{{\mathcal H}(n)}(\tau_n(x\cdot e_1-\ell)\leq\tau_n(x\cdot e_1+\ell/3)|X_n=x ) \leq \frac{C_{x,\partial B\setminus B^+}}{C_{x,B^+}} , 
\end{align*}
where $C_{x,A}$ denoted the effective conductance between a point $x$ and a set $A$ (see also Fact 2 in \cite{BGP2003speed}). The conductance $C_{x,B^+}$ is bounded from below by the conductance of a good path from $x$ to $B^+$, which is at least $D_1e^{\lambda 2 x\cdot e_1} $ for some $D_1=D_1(\lambda)$. Furthermore, we have the upper bound
\begin{align*}
C_{x,\partial B\setminus B^+} \leq C_{x,\partial B^-}+C_{x,\partial B_1}+C_{x,\partial B_2} , 
\end{align*}
where 
\begin{align*}
B^- & = x + \{-\ell \}\times [ - e^{\lambda \ell} ,  e^{\lambda \ell}  ] , \\
B_1 & = x + [-\ell , \ell/3]\times \{ - e^{\lambda \ell} \} , \\
B_2 & = x + [-\ell , \ell/3]\times \{ e^{\lambda \ell}  \} . 
\end{align*}
The effective conductance $C_{x,\partial B^-}$ is bounded from above by the sum of the edge weights between $z$ and $z+e_1$, for $z\in B^-$. But for every such $z$, the weight is 
\begin{align*}
\omega(z,z+e_1)e^{\lambda(2z\cdot e_1+1)} \leq e^{\lambda (2x\cdot e_1 -2\ell+1)} .
\end{align*}
There are at most $2e^{\lambda \ell} +1$ such edges. Therefore $ C_{x,\partial B^-}\leq D_2 e^{ \lambda(2x\cdot e_1 -\ell)}$ for some $D_2=D_2(\lambda)$. Finally, the Nash-Williams inequality gives
\begin{align*}
C_{x,B_i} \leq e^{-\lambda \ell} \sum_{i={x\cdot e_1 -\ell}}^{x\cdot e_1+\ell/3} e^{2\lambda (i+1)}
\leq D_3  e^{\lambda (2x\cdot e_1 -\ell/3)}
\end{align*}
for some $D_3=D_3(\lambda)$. Combining the bounds for the effective conductances, we get the desired bound for the exit probability. 
\hfill $\Box$

Let $G(x)$ be the event that $x$ is a good point. We call a time point $n$ a fresh epoch, if $(X_n-X_k)\cdot e_1>0$ for all $k<n$ and let $F(n)$ be the event that $n$ is a fresh epoch. From the bound in Lemma \ref{lem:preconduc}, we get the following inequalities (Lemma 3 and Lemma 4 in \cite{BGP2003speed}). In the following, take $p$ so close to 1 that $\alpha(p)$ in Lemma \ref{lem:trapsize} is less than 1.
Then there exists a constant $D=D(\lambda,p)$ such that 
\begin{align}\label{probbackstepping}
\mathbb{P}_{\lambda}^{{\mathcal H}(n)} ( \text{ there is an } m\geq n \text{ such that } (X_m-X_n)\cdot e_1\leq -\ell |\, F(n),G(X_n) ) \leq D e^{-\lambda \sqrt{\ell } /D} , \,  \mathbb{P}_\lambda-a.s. 
\end{align}
Let $\tau_n'(h)$ be the first fresh epoch later than $n$, such that the random walk hits a good point whose first coordinate is larger or equal to $h$. Then, there exists a constant $K=K(\lambda,p)$ such that for any $\ell\geq 1$
\begin{align} \label{probbackstepping2}
\mathbb{P}_{\lambda}^{{\mathcal H}(n)} \left( \tau(X_n\cdot e_1-\ell) < \tau_n'( X_n\cdot e_1 + \ell/6) \left| \, G(X_n), \max_{0\leq i\leq n} (X_i-X_n)\cdot e_1 <\sqrt{\ell} \right. \right) \leq K e^{-\lambda \sqrt{\ell } /K} ,  
\end{align}
$\mathbb{P}_\lambda$-almost surely.  
In particular, 
\begin{align} \label{probbackstepping3}
\mathbb{P}_{\lambda}^{{\mathcal H}(n)} \left( \tau(X_n\cdot e_1-\ell) < \tau_n'( X_n\cdot e_1 + \ell/6) |\, F(n), G(X_n) \right) \leq K e^{-\lambda \sqrt{\ell } /K} ,  
\end{align}
$\mathbb{P}_\lambda$-almost surely. From these bounds, the following lower bound for the speed is proven. Note that the constant is independent of $\kappa$. 

\begin{lem} \label{lem:displacement}
For $\lambda$ sufficiently small, there exists a constant $C=C(p)$ such that
\begin{align*}
\mathbb{P}_{\lambda} (X_n\cdot e_1 < C n^{1/10}) \leq C n^{-2} .
\end{align*}
\end{lem}

Let us highlight the only change necessary in the proof given in \cite{BGP2003speed}: Therein, the Carne-Varopoulos bound 
\begin{align} \label{CVbound}
P_{\omega,\lambda}^x(X_n=y) \leq 2 \sqrt{\frac{\pi(y)}{\pi(x)}} \exp \left( -\frac{d(x,y)^2}{2n} \right)
\end{align} 
is applied, with $\pi$ the reversible measure and $d(\cdot,\cdot)$ the graph distance. On the percolation cluster, it is easy to get a further upper bound, since in this case,
\begin{align*}
e^{\lambda (2x\cdot e_1-1)} \leq \pi(x)\leq 4 e^{\lambda (2x\cdot e_1+1)} , 
\end{align*} 
as every point $x$ in the cluster is the endpoint of an edge with conductance 1. Of course, the upper bound is still valid in our case, but the lower bound depends on $\kappa$ if $x$ is surrounded by only $\kappa$-edges. To get a lower bound independent of $\kappa$, let $J(x)$ be the connected component of points surrounded by $\kappa$-edges. If $J(x)$ is empty, we can proceed as in the percolation case. Otherwise, let 
\begin{align*}
T_x = \inf\{n\geq 0 : X_n\notin J(x)\}
\end{align*}
and
define for positive integers $d_n$ the events
\begin{align*}
A_n = \{ \operatorname{diam}(J(z))\leq d_n \text{ for all } z-x\in [-n,n]^2 \} , 
\end{align*} 
then by Lemma \ref{lem:trapsize}, 
\begin{align} \label{trapdiameter} 
P(A_n^c) \leq n^2 P(\operatorname{diam}(T(0))> d_n ) \leq 2n^2 \alpha^{d_n/2} .
\end{align}
For an environment $\omega \in A_n$ we have then for the hitting probability 
\begin{align} \label{kappatraps}
P_{\omega,\lambda}^x(X_n=y) 
& \leq \sum_{z\in \mathbb{Z}^2} \sum_{1\leq m\leq n} P_{\omega,\lambda}^z(X_{n-m}=y)P_{\omega,\lambda}^z(T_x=m, X_{T_x}=z) 
%& \leq P_{\omega,\lambda}^z(T_x\geq M) +  \sum_{z\in \mathbb{Z}^2} \sum_{m\geq 1} P_{\omega,\lambda}^z(X_{n-m}=y)P_{\omega,%%\lambda}^z(T_x=m, X_{T_x}=z) .
\end{align}
On $A_n$, there are at most $d_n^2$ points $z$ such that the second probability in the sum is nonzero, and for each such $z$ we have by the Carne-Varopoulos bound
\begin{align*}
P_{\omega,\lambda}^z(X_{n-m}=y) \leq 4e^{\lambda ((y-z)\cdot e_1+1)} \exp \left( -\frac{d(z,y)^2}{2(n-m)} \right)
\leq 4e^{\lambda ((x-y)\cdot e_1+d_n+1)} \exp \left( -\frac{(d(x,y)-d_n)^2}{2n} \right) .  
\end{align*}
Let $d_n=\gamma \log(n)$ for $\gamma =-8/\log(\alpha)$, then for all but finitely many $n$, $A_n$ occurs. For all $\omega\in A_n$ and $1\leq i<j\leq n$ we may conclude by the union bound
\begin{align*}
P_{\omega,\lambda}^x ( X_i\cdot e_1=X_j\cdot e_1 \text{ but } ||X_i-X_j||\geq n^{6/10}) & \leq 4 n^4 d_n^2 e^{\lambda (d_n+1)} \exp\left( -\frac{(n^{6/10}-d_n)^2}{2n} \right) \\ 
& \leq \exp\left( -\frac{1}{5}n^{1/10}\right)  
\end{align*}
for $n$ sufficiently large, 
which yields the necessary estimate in \cite{BGP2003speed}. 

\begin{lem} \label{lem:neverreturn}
There exists a constant $c=c(\lambda, p)>0$ such that
\begin{align*}
\mathbb{P}_\lambda (X_n\cdot e_1 \geq 1 \text{ for all } n\geq 1) >c .
\end{align*}
\end{lem}

\textbf{Proof:} 
Let $\ell_0=N$ be a positive integer and $\ell_{i+1}=13 \ell_i/12$ for $i\geq 1$. Define recursively the times $t_0=N$, $t_{i+1}=\tau_{t_i}'(X_{t_i}\cdot e_1+\ell_i/6)$ and the events
\begin{align*}
A_0 = \{ X_N = (N,0) \text{ and } (N,0) \text{ is a good point } \}
\end{align*}
and
\begin{align*}
A_i = \{ \tau_{t_i}'(X_{t_i}\cdot e_1+\ell_i/6) < \tau_{t_i}(X_{t_i}\cdot e_1-\ell_i) \} . 
\end{align*}
Then $\mathbb{P}_\lambda (A_0)=c_N>0$ and by \eqref{probbackstepping3},
\begin{align*}
\mathbb{P}_\lambda (A_i^c)\leq K e^{-\lambda \sqrt{\ell_i } /K} .
\end{align*} 
Therefore,
\begin{align*}
\mathbb{P}_\lambda \left( \bigcap_{i=0}^\infty A_i \right) \geq  c_N(1-Ce^{-\lambda \sqrt{N}/K }) ,
\end{align*}
which is positive for $N$ large enough. When all of the events $A_i$ occur, then $t_i<\infty$ for all $i$ and if $m\geq t_i$,
\begin{align*}
X_m\cdot e_1 > X_{t_i}\cdot e_1-\ell_i \geq X_{t_0}\cdot e_1-\ell_0 + \frac{1}{12}\sum_{j=1}^{i-1} \ell_j \geq \frac{N}{12} \left( \frac{13}{12}\right)^{i-1} ,
\end{align*} 
which implies in particular $X_n\cdot e_1 \geq 1 \text{ for all } n\geq 1$. 
\hfill $\Box$

We now introduce a regeneration structure, slightly different from the one used to prove Theorem \ref{thm:monotone}. Recall that $n$ is a fresh epoch, if $X_n\cdot e_1>X_k\cdot e_1$ for all $k<n$. If $n$ is a fresh epoch and additionally, $X_n\cdot e_1<X_k\cdot e_1$ for all $k>n$, we call $n$ a regeneration and we denote by $R_n$ the $n$-th regeneration time.

For $z\in \mathbb{Z}^2$, let $\omega_z^+= \{\omega_z(x,y) :\, x\sim y, x\cdot e_1\geq z\cdot e_1 \}$ be the environment to the right of $z$. The following lemma is standard in the theory of random walks in random environments, see 
\cite{sznizer1999}.

\begin{lem} \label{lem:regenerations2}
The sequence 
\begin{align*}
\big( (X_{R_n+k}-X_{R_n})_{k\geq0}, \omega_{R_n}^+ \big)_{n\geq 1}
\end{align*}
is stationary and ergodic. Moreover, the distribution of $((X_{R_n+k}-X_{R_n})_{k\geq0}, \omega^+_{R_n})$ is given by the distribution of $((X_k)_{k\geq0}, \omega^+_{0})$ under $\mathbb{P}_\lambda$, conditioned on $\{ X_n\cdot e_1 \geq 1 \text{ for all } n\geq 1\}$.
\end{lem}

It follows from Lemma \ref{lem:regenerations2} that $v(\lambda)$ exists and is nonzero if and only if $\mathbb{E}_\lambda [R_2-R_1]<\infty$ and in this case
\begin{align} \label{speedperc}
v(\lambda) = \frac{\mathbb{E}_\lambda [X_{R_2}-X_{R_1}]}{\mathbb{E}_\lambda [R_2-R_1]} .
\end{align}
Since $(X_{R_2}-X_{R_1})\cdot e_1\geq 1$, the inequality \eqref{nonmonotonicitylowerbound} follows then from 
\begin{align*}
\mathbb{E}_\lambda [R_2-R_1]\leq C
\end{align*}
with a constant $C=C(\lambda,p)$ independent of $\kappa$. This inequality follows by the same arguments as Lemma 8 in \cite{BGP2003speed}, making use of Lemma \ref{lem:preconduc}, Lemma \ref{lem:displacement} and Lemma \ref{lem:neverreturn}.

\subsubsection{An upper bound for $v_1(\lambda_2)$}

The upper bound \eqref{nonmonotonicityupperbound} follows from the fact that for small values of $\kappa$, the random walk will spend a long time in dead ends of the percolation cluster. To be more precise, let $I(x)$ be the connected component of $x$ connected by open edges (i.e., with conductance 1). We call $x\in \mathbb{Z}^2$ the beginning of a dead end, if $x$ belongs to the infinite cluster to its left, but not to the infinite cluster to its right, i.e., $I(x)\cap \{z:\, (z-x)\cdot e_1<0\}$ is infinite but $I(x)\cap \{z:\, (z-x)\cdot e_1\geq 0\}$ is finite. The dead end starting at $x$ is the finite set $I(x)\cap \{z:\, (z-x)\cdot e_1\geq 0\}$. Let $A$ be a dead end starting at the origin and $d(A)=\max\{z\cdot e_1:\, z\in A\}$ the depth of $A$. The time spent in $A$ will be denoted by
\begin{align} \label{timeintraps}
T_A = \inf\{n\geq 1:\ X_{n}\cdot e_1\leq 0 \} . 
\end{align}
If there is no dead end at the origin, set $A=\emptyset, d(A)=0$ and $T_A=0$. For an environment $\omega$ with $\omega(x,y)\in \{\kappa,1\}$ for $x\sim y$, let $\bar \omega$ be the environment obtained from $\omega$ by setting $\kappa=0$. We use the coupling introduced in Section \ref{sec:coupling} and denote by $(\bar{X}_n)_n$ the random walk in the environment $\bar \omega$. 
It was shown in \cite{BGP2003speed}, that there exists a $\lambda_u<\infty$, such that for $\lambda>\lambda_u$, $\mathbb{E}_\lambda [\bar T_A]=\infty$, when $\bar T_A$ is the time $\bar X_n$ spends in $A$. In the following, fix such a $\lambda$. We claim that
\begin{align}\label{timeintraps2}
\lim_{\kappa \to 0} \mathbb{E}_\lambda[T_A] =  \infty . 
\end{align}
Indeed, as in \eqref{couplingerror},
\begin{align} \label{couplingerror3}
P_{\omega,\lambda} (X_n-X_{n-1} \neq \bar X_n-\bar X_{n-1} |\, X_{n-1}= \bar X_{n-1}= x ) \leq C \kappa 
\end{align}
for all $n\geq 1$ and $x \in \mathbb{Z}^2$. Let
\begin{align*}
D = \inf\{ n \geq 1:\, X_n-X_{n-1} \neq \bar X_n-\bar X_{n-1} \} . 
\end{align*}
Since \eqref{couplingerror3} holds independent of $x$, $D$ can be coupled with a geometric distributed random variable $G$ with mean $(C\kappa)^{-1}$ independent of $T_A$ such that $D\geq G$. Therefore, 
\begin{align*}
\mathbb{E}_\lambda[T_A] \geq \mathbb{E}_\lambda[\bar T_A \wedge D] \geq \mathbb{E}_\lambda[\bar T_A \wedge G] \xrightarrow[\kappa\to 0]{} \mathbb{E}_\lambda[\bar T_A ]=\infty .
\end{align*}
Next, we define a sequence of ladder times $L_0,L_1,\dots $ with $L_0=0$ and let $A_0$ be the dead end starting at the origin (possibly empty). Inductively, let $L_{i+1}$ be the first fresh epoch with $X_{L_{i+1}}\cdot e_1> X_{L_i}\cdot e_1 + d(A_i)$ and let $A_{i+1}$ be the dead end beginning at $X_{L_{i+1}}$. Since $X_n$ is transient to the right, there are infinitely many ladder times. Note that $L_{i+1}-L_i\geq T_{A_i}$ and the random variables $T_{A_1},T_{A_2},\dots $ are iid under $\mathbb{P}_\lambda$ and satisfy \eqref{timeintraps2}. Additionally, the random variables $X_{L_{i+1}}\cdot e_1- X_{L_{i}}\cdot e_1=d(A_i)+1$ are iid and have exponential moments (independent of $\kappa$) by Lemma \ref{lem:trapsize}. This implies for the speed
\begin{align*}
v_1(\lambda) = \lim_{n\to \infty} \frac{X_{L_n}\cdot e_1}{L_n} \leq \lim_{n\to \infty} \frac{\sum_{i=0}^n d(A_i)+1}{\sum_{i=0}^n T_{A_i}} \leq \frac{C}{\mathbb{E}_\lambda[T_A]} .
\end{align*}
Letting $\kappa \to 0$, we obtain \eqref{nonmonotonicityupperbound} by \eqref{timeintraps2}. This completes the proof of Theorem \ref{thm:nonmonotone}.

{\bf Acknowledgement}
We thank Andrew Barbour for helpful discussions about the proof of 
Lemma \ref{lem:deltato0}.
Support of DFG (grant GA 582/8-1) is gratefully acknowledged.

\bigskip

\bibliographystyle{alpha}
\bibliography{biber}

{\footnotesize

Noam Berger: Technische Universit\"at M\"unchen,
Fakult\"at f\"ur Mathematik,
Boltzmannstra\ss e~3, 
85748~Garching bei M\"unchen,
Germany, noam.berger@tum.de\\

Nina Gantert: Technische Universit\"at M\"unchen,
Fakult\"at f\"ur Mathematik,
Boltzmannstra\ss e~3, 
85748~Garching bei M\"unchen,
Germany, gantert@ma.tum.de\\

Jan Nagel: 
Eindhoven University of Technology,
Department of Mathematics and Computer Science,
P.O. Box 513,
5600 MB Eindhoven, 
the Netherlands,
J.H.Nagel@tue.nl\\  
}

\end{document}